\documentclass[preprint,12pt]{elsarticle}

\usepackage[utf8]{inputenc}
\usepackage[english]{babel}

\usepackage{verbatim} 
\usepackage[normalem]{ulem} 

\usepackage{fullpage}
\usepackage{setspace}
\usepackage{amsthm}   
\usepackage{amsmath}  
\usepackage{amssymb}  
\usepackage{graphicx}
\usepackage{datetime}
\usepackage{enumitem} 
\usepackage{mathtools}
\usepackage{array}
\usepackage{bm}
\usepackage{caption}
\usepackage{subcaption}
\usepackage{empheq}
\usepackage[mathscr]{eucal}
\usepackage{xspace} 

\newtheorem{theorem}              {Theorem}
\newtheorem{lemma}      [theorem] {Lemma}

\newtheorem{corollary}  [theorem] {Corollary}

\newtheorem{remark}{Remark}
\newtheorem{definition} [theorem] {Definition}

\def\alabel{\upshape({\itshape \alph*\,})}

\usepackage{tikz}  
\usetikzlibrary{shapes}
\usetikzlibrary{arrows}
\usetikzlibrary{calc}
\usetikzlibrary{decorations.markings}
\usetikzlibrary{arrows.meta}
\usetikzlibrary{shapes.misc}   
\usetikzlibrary{positioning}
\usetikzlibrary{fit}
\usetikzlibrary{external} 
\tikzset{non code/.style={circle, draw=black}}
\tikzset{code/.style={circle, fill=black}}
\tikzset{vtx/.style={circle, draw=black, thick}}

\usepackage{dsfont} 
\usepackage[babel]{microtype}

\usepackage{xcolor} 
\usepackage{hyperref}

\usepackage[T1]{fontenc}

\makeatletter
\def\moverlay{\mathpalette\mov@rlay}
\def\mov@rlay#1#2{\leavevmode\vtop{   \baselineskip\z@skip \lineskiplimit-\maxdimen
   \ialign{\hfil$\m@th#1##$\hfil\cr#2\crcr}}}
\newcommand{\charfusion}[3][\mathord]{
    #1{\ifx#1\mathop\vphantom{#2}\fi
        \mathpalette\mov@rlay{#2\cr#3}
      }
    \ifx#1\mathop\expandafter\displaylimits\fi}
\makeatother

\newcommand{\quot}[2]{\mathchoice%
{\left.\raisebox{.1em}{$\displaystyle{#1}$}\kern-1pt/\raisebox{-.2em}{$\displaystyle{#2}$}\right.}
{\left.\raisebox{.1em}{${#1}$}\kern-1pt/\raisebox{-.2em}{${#2}$}\right.}
{\left.\raisebox{.1em}{$\scriptstyle{#1}$}\kern-1pt/\raisebox{-.2em}{$\scriptstyle{#2}$}\right.}
{\left.\raisebox{.1em}{$\scriptscriptstyle{#1}$}\kern-1pt/\raisebox{-.2em}{$\scriptscriptstyle{#2}$}\right.}
}

\DeclareFontFamily{U}  {MnSymbolC}{}
\DeclareSymbolFont{MnSyC}         {U}  {MnSymbolC}{m}{n}
\DeclareFontShape{U}{MnSymbolC}{m}{n}{
    <-6>  MnSymbolC5
   <6-7>  MnSymbolC6
   <7-8>  MnSymbolC7
   <8-9>  MnSymbolC8
   <9-10> MnSymbolC9
  <10-12> MnSymbolC10
  <12->   MnSymbolC12}{}
\DeclareMathSymbol{\powerset}{\mathord}{MnSyC}{180}

\let\epsilon\varepsilon
\let\smallhat\hat
\let\hat\widehat

\allowdisplaybreaks[2]

\newcommand{\lds}{\textup{\textsc{lds}}\xspace} 
\DeclareMathOperator{\dist}{dist}  
\newcommand{\NP}{$\mathrm{NP}$} 
\newcommand{\MinDen}{\textup{\textsc{Min-density}}}
\newcommand{\MinCard}{\textup{\textsc{Min-cardinality}}}

\begin{document}


\begin{frontmatter}
  




\title{The complexity of minimum-density locating-dominating set  in\\ infinite
  periodic graphs\tnoteref{1}}

\tnotetext[1]{Research partially supported by CNPq (Proc.~311892/2021-3) and FAPESP (Proc.~2023/03167-5)}
\author[a]{Arthur C. Gomes} 
\ead{arthurcgomes@ime.usp.br}

\author[a]{Yoshiko Wakabayashi\corref{cor1}}
\cortext[cor1]{Corresponding author}
\ead{yw@ime.usp.br}

\affiliation[a]{organization={University of S\~ao Paulo, Institute of Mathematics, Statistics and
    Computer Science},
             addressline={Rua do Mat\~ao, 1010 - Cidade Universit\'aria}, 
             city={S\~ao Paulo},
             postcode={05508-090}, 
             state={SP},
             country={Brazil}}

\begin{abstract}
  A dominating set $S$ of a graph $G$ is a \emph{locating-dominating
    set} (\lds) if, for each pair of distinct vertices not in~$S$,
  their neighbourhoods in $S$ are distinct. Finding a
  minimum-cardinality \lds  in finite graphs is a well-known NP-hard
  problem.  On infinite graphs, this problem naturally generalises to
  finding an \lds of minimum density. While density bounds have been
  widely studied for specific infinite regular grids, no computational
  complexity results exist for infinite graphs. We prove that the
  minimum-density \lds problem in infinite $\mathbb{Z}$-periodic
  graphs with a finite period is NP-hard.  This result bridges the gap
  between cardinality minimization on finite graphs and density
  minimization on infinite graphs via a rigorous periodic reduction.
  Furthermore, our approach can be adapted to establish NP-hardness
  for related structural problems on infinite periodic graphs.
\end{abstract}


\end{frontmatter}


\section{Introduction}
\label{sec:introduction}

Let $G = (V, E)$ be a connected graph, and let $\text{dist}(u, v)$
denote the distance between two vertices~$u, v \in V$. We say that
vertices at distance~1 are \emph{neighbours}. The \emph{(open)
  neighbourhood} of a vertex $v \in V$, denoted $N(v)$, is defined as
$N(v) \coloneqq \{u \in V : \dist(v, u) = 1\}$, while the \emph{closed
  neighbourhood} of~$v$, denoted $N[v]$, is defined as
$N[v] \coloneqq N(v) \cup \{v\}$.  A \emph{dominating set} of $G$ is a
set~$S \subseteq V$ such that each vertex $u \in V \setminus S$ has a
neighbour in $S$. A dominating set $S$ of $G$ is a
\emph{locating-dominating set} (\lds) if $S$ has the additional
\emph{locating property}: each vertex $v \in V \setminus S$ can be
uniquely distinguished by its neighbourhood in $S$.  Formally, this
means that for each pair of distinct vertices
$u, v \in V \setminus S$, we must have $N(u) \cap S \neq N(v) \cap S$.
Locating-dominating set were introduced by
Slater~\cite{Slater75,Slater02} in the early seventies; they are
used to model and optimize diagnostic systems where sensor
placement must minimize costs while retaining precise tracking
capabilities.

The studies of this concept in finite graphs have focused primarily in
the problem of finding an \lds of minimum cardinality. We refer to
such a set as a \emph{minimum~\lds} and denote its cardinality in a
graph $G$ as $\gamma^{LD}(G)$. Finding a minimum \lds is NP-hard for
arbitrary graphs, but it is polynomially solvable when restricted to
graph classes with specific structural properties (such as trees,
series-parallel networks, outerplanar graphs and others).  It remains
NP-hard even when restricted to specific graph classes such as
bipartite graphs~\cite{CharonHL03}, interval
graphs~\cite{FoucaudMNPV16}, and subcubic planar bipartite
graphs~\cite{Foucaud15}. Furthermore, the problem is $\log$-APX-hard
for general finite graphs and remains so for several graph
classes~\cite{Foucaud15,Suomela07}.  On infinite graphs, the objective
generalises to finding an \lds of minimum density, and research has
focused on regular grids, such as square, triangular, king, and
hexagonal grids~\cite{Slater02,HonkalaL06,Honkala06}.  More recently,
infinite grids with a finite number of rows (known as strips of finite width) 
have also been investigated~\cite{Junnila15,BouznifDMP19,GomesW25Procedia,GomesW26a}.

Before defining the density of a set in a graph, let us
first establish some notation.

Let $G=(V,E)$ be a graph and let $r \geq 1$ be a natural number. For a
vertex $v$ in $G$, we define the 
\emph{$r$-open neighbourhood} of $v$ in $G$ as the set
$N_r(v) \coloneqq \{w \in V(G) : 0 < \dist(v, w) \leq r\}$.

The \emph{density} of a set $S \subseteq V$, denoted $d(S, G)$, is defined as
\begin{align*}
    d(S, G) \coloneqq \inf \{d_w(S, G) : w \in V\}, \quad \text{where} \quad 
    d_w(S, G) \coloneqq \limsup_{r\to\infty} \frac{|N_r(w) \cap S|}{|N_r(w)|}.
\end{align*}

The \emph{minimum density of an \lds of a graph $G$}, denoted $d^*(G)$, is defined as
\[
 d^*(G) \coloneqq \inf \{d(S, G) : S \text{ is an \lds of } G\}.
\]



\section{Preliminaries for the reduction} 
\label{sec:he}

The following two problems will be central to our studies: the
\MinCard{} \lds and the \MinDen{} \lds.  The first one seeks an \lds
of minimum cardinality in a finite graph, and the second one seeks an
\lds of minimum density in an infinite graph.

While the first one is known to be
NP-hard~\cite{CharonHL03,Foucaud15}, to the best of our knowledge, the
computational complexity of the \MinDen{} \lds problem in infinite
graphs or infinite strips (of a finite width) has not been studied in
the literature, although many results on this topic have been obtained
for regular infinite grids.

In the next section, we show a polynomial-time reduction from the \MinCard{} 
\lds problem to the \MinDen{} \lds problem in infinite $\mathbb{Z}$-periodic 
graphs with a finite period. 
More formally, we reduce the first problem on a finite graph~$G$ to the second 
problem on an infinite graph $G^{\infty}$ whose automorphism group contains a 
subgroup isomorphic to~$\mathbb{Z}$ with a finite quotient.
This concept will become clear when we define the infinite graph used in the reduction.
To present the reduction, we need some preliminary definitions and
results, which we give in this section.

Given an arbitrary finite graph $G$, we construct first a finite graph
$H = H(G)$, defined below, which contains $G$ as a subgraph.  This
graph $H(G)$ is used later in Section~\ref{sec:G-infinite} in the
construction of the infinite periodic graph $G^{\infty}$ mentioned
above.



  \begin{definition}
    [$\mathbf{G}$-\textbf{gadget}]\label{def:gadget}
  Let $G$ be a finite connected graph with $n\geq 5$ vertices, and
  let~$z$ be an arbitrary vertex in~$G$.  Let $B$ be a graph called
  \textbf{Butterfly}, with vertex set 
  $V(B) =\{a, \smallhat{a}, b, \smallhat{b}, c, \smallhat{c}, d,
  \smallhat{d}, r, s, t, u, v, w, x \}$, and edge set
  $E(B) = \{ a\smallhat{a}, as, sb, b\smallhat{b}, bt, tx, xr, ra,
  c\smallhat{c}, cv, vd,\break d\smallhat{d}, dw, wx, xu, uc\}$.  Let $P$ be
  a path defined by the sequence of vertices
  $(v_1,v_2,\ldots, v_{5n})$. Consider that the graphs $G$, $B$ and $P$ are all
  pairwise vertex-disjoint.

  Let $H = H(G)$ be the graph, called \textbf{G-gadget}, whose vertex
  set is $V(H) \coloneqq V(G) \cup V(B) \cup \{y\} \cup V(P)$, where $y\notin V(G)$,
  and edge set is $E(H) \coloneqq E(G) \cup E(B) \cup \{xy, yz, dv_1\} \cup E(P)$. 
  Note that $|V(H)|= 6n + 16$,  $|V(G)|=n$,  $|V(B)|=15$ and $|V(P)|=5n$.
\end{definition}

In Figure~\ref{fig:gadget} we exhibit the $G$-gadget $H(G)$, where $G$
is the $5$-vertex graph within the dashed cycle, with a chosen vertex
$z$ (as indicated).

\begin{remark}
  For a finite graph $G$, when we refer to the $G$-gadget $H(G)$, we
  are assuming that an arbitrary vertex $z$ in $G$ has been chosen.
  Thus, throughout this text, for a same graph~$G$, all
  gadgets $H(G)$ are isomorphic.
  \end{remark}
  \vspace{-3mm}
\begin{figure}[ht] 
    \centering
    \includegraphics{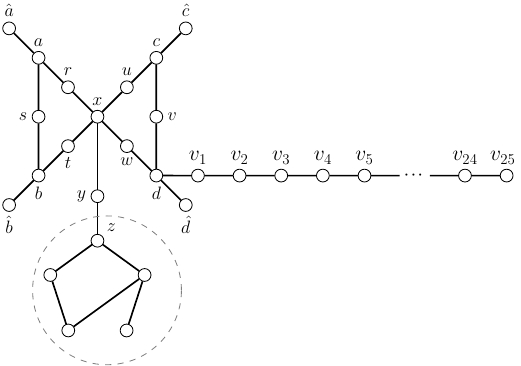}
    \caption{The $G$-gadget $H(G)$ for the graph $G$ indicated in 
    the dashed cycle.}
    \label{fig:gadget}
\end{figure}


Before using $H(G)$ in the construction of an infinite graph
$G^{\infty}$, we prove a lemma that will be helpful to understand how
a minimum \lds of $H(G)$ is related to a minimum \lds of $G$.

\begin{lemma}    
    \label{lemma:lds-butterfly}
    Let $G$ be an $n$-vertex connected graph and  $H=H(G)$ be the
    $G$-gadget as in Definition~\ref{def:gadget}. Let $\hat{H}$ be the
    graph obtained from $H$ after removing the graph $G$ and the
    vertex~$y$. Let $S$ be a minimum \lds of $H$ and
    $\hat{S} \coloneqq S \cap V(\hat{H})$. Then, the intersection of
    $\hat{S}$ with the Butterfly is precisely the set
    $\{a, b, c, d, x\}$.  Moreover, $\hat{S}$ is a minimum \lds of
    $\hat{H}$, and has size~$2n+5$.
\end{lemma}

\begin{proof}   
    Let $S$ be a minimum {\lds} of $H$, and let
    $\hat{S} \coloneqq S \cap V(\hat{H})$.  Since the vertices $\smallhat{a}$,
    $\smallhat{b}$, $\smallhat{c}$, $\smallhat{d}$ have degree one in
    $H$, it follows that at least one vertex in each of the~4 disjoint
    sets $\{a, \smallhat{a}\}$, $\{b, \smallhat{b}\}$,
    $\{c, \smallhat{c}\}$, $\{d, \smallhat{d}\}$ must be in $\hat{S}$. 
    A result of Bertrand et al.~\cite{BertrandCHL04} says that any \lds{}
    of a path must contain at least two vertices out of each group of
    5~consecutive vertices. Thus, $\hat{S}$ must contain at least $2n$
    vertices of the path $P=(v_1, v_2,\ldots, v_{5n})$, regardless of
    $d$ being in $\hat{S}$ or not.

    Let $S^* \coloneqq \{a, b, c, d, x\} \cup \big\{ v_{5j+2}, v_{5j+4} : j = 0, \ldots,
    n-1 \big\}$.  Clearly,  $S^*$ is an \lds{} of $\hat{H}$ of size $2n + 5$ (see
    Figure~\ref{fig:lds-H-hat}). Thus,  $\gamma^{LD}(\hat{H}) \leq 2n +
    5$. Let us prove that $\gamma^{LD}(\hat{H}) = 2n + 5$.

    We claim that $\hat{S}$ restricted to the Butterfly is precisely
    the set $\{a, b, c, d, x\}$. (Note that the Butterfly is symmetric
    with respect to the vertex $x$.) Suppose $a \notin \hat{S}$. Then
    $\smallhat{a} \in \hat{S}$. (i) Suppose $b \notin \hat{S}$. Then
    $\smallhat{b}\in \hat{S}$. Since the two neighbours of $s$ are not
    in~$\hat{S}$, we must have $s\in \hat{S}$. In this case, $\hat{S}$
    will contain more than~5 vertices of the Butterfly; and therefore
    $|\hat{S}| > 2n + 5$, a contradiction (because
    $(S\setminus \hat{S}) \cup S^*$ would be an {\lds} of $H$ smaller
    than $S$).  (ii) Thus, $b\in \hat{S}$. In this case,
    $N(\smallhat{b}) \cap \hat{S} = N(s) \cap \hat{S} = \{b\}$, and
    therefore at least one among~$\smallhat{b}$ and~$s$ must also be
    in~$\hat{S}$, implying that $\hat{S}$ will contain more than~5
    vertices of the Butterfly, again a contradiction.  Therefore,
    $a \in \hat{S}$.  In this case, if $b \notin \hat{S}$ then
    $N(\smallhat{a}) \cap \hat{S} = N(s) \cap \hat{S} =\{a\}$,
    implying that either $\smallhat{a}$ or $s$ must belong to
    $\hat{S}$. Moreover, since $\smallhat{b}\in \hat{S}$, we conclude
    that $\hat{S}$ contains more than~5 vertices of the Butterfly, a
    contradiction. By symmetry, we conclude that $c\in \hat{S}$ and
    $d\in \hat{S}$. Thus, $\hat{S} \supset \{a,b,c,d\}$.

    If $x \notin \hat{S}$, then at least one vertex of each of the
    four pairs $\{\smallhat{a}, r\}$, $\{\smallhat{b}, t\}$,
    $\{\smallhat{c}, u\}$ and $\{\smallhat{d}, w\}$ must be in
    $\hat{S}$, because~$S$ has the locating property. But then,
    $\hat{S}$ will have at least $8$ vertices of the Butterfly, and
    therefore $\hat{S}$ will have size at least $2n+8$, a contradiction
    to the minimality of~$S$. Hence, $\hat{S} \supset
    \{a,b,c,d,x\}$. Since $\{a,b,c,d,x\}$ is an {\lds} of the Butterfly,
    it is the unique choice in the Butterfly that together with other
    $2n$ vertices in the path $P$ defines a minimum {\lds} of
    $\hat{H}$. This concludes the proof that $\hat{S}$ is a minimum
    {\lds} of $\hat{H}$ and $|\hat{S}| = 2n+5$.
  \end{proof}

  \vspace{-3mm} 
\begin{figure}[ht]   
    \centering
    \includegraphics{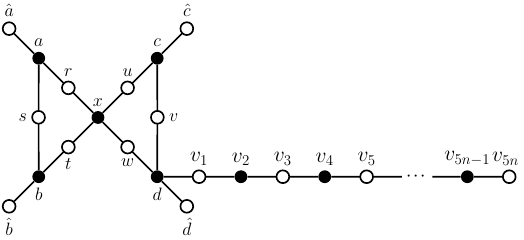}
    \caption{An \lds{} of size $2n + 5$ of $\hat{H}$, where $G$ has
      $n$ vertices.} 
    \label{fig:lds-H-hat}
\end{figure}

In Lemma~\ref{lemma:lds-butterfly}, we showed that every minimum \lds of $H$ 
contains the vertex $x$. 
It remains to analyse the role of the vertex $y$ in a minimum \lds of 
$H = \hat{H} \cup G \cup \{xy, yz\}$.

\begin{lemma}   
  \label{lemma:join-graphs}
  Let $G$ be an $n$-vertex connected graph and $H = H(G)$ be the
  $G$-gadget as in Definition~\ref{def:gadget}.  Let $S$ be a minimum
  \lds of $H$. Then $|S| = 2n + 5 + \gamma^{LD}(G)$.  Moreover, the
  following holds:
    \begin{enumerate}[noitemsep,topsep=0.5mm,label=\alabel]
        \item If $y \notin S$, then $S \cap V(G)$ is a minimum {\lds} of $G$.
        \item If $y \in S$, then $Z \coloneqq (S \setminus \{y\}) \cup \{z\}$ 
            is also a minimum {\lds}  of $H$, and $Z \cap V(G)$ is a minimum 
            {\lds} of $G$. 
    \end{enumerate}
\end{lemma}  

\begin{figure}[ht]  
    \centering
    \begin{minipage}[b]{0.5\linewidth}
        \centering
        \scalebox{0.82}{\includegraphics{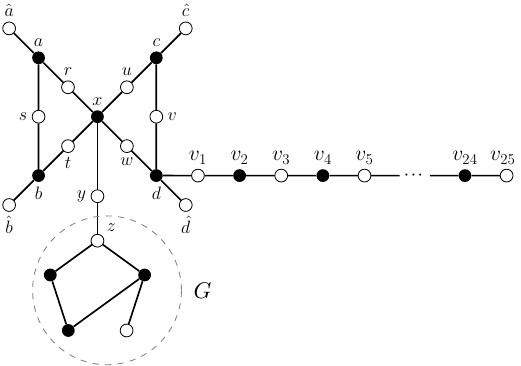}}
    \end{minipage}
    \begin{minipage}[b]{0.5\linewidth}
        \centering
        \scalebox{0.82}{\includegraphics{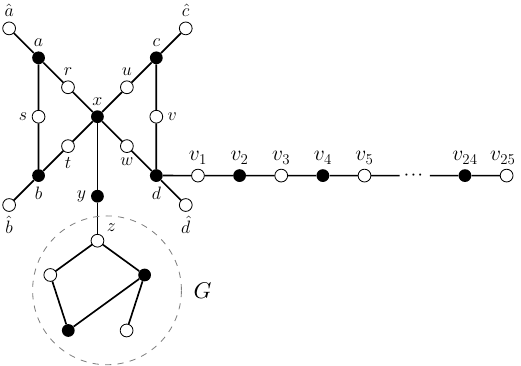}}
    \end{minipage}
    \caption{Optimal locating-dominating sets for $H(G)$: (a) not containing $y$;  (b)
      containing $y$.}
    \label{fig:two-solutions}
  \end{figure}
\begin{proof}
    Let $S$ be a minimum {\lds} of $H$ and let $\hat{S} \coloneqq S\cap
    V(\hat{H})$.

    \smallskip
    \noindent \textbf{Case (a)} When  $y \notin S$.
    By Lemma~\ref{lemma:lds-butterfly}, $\hat{S}$ is a minimum {\lds} of
    $\hat{H}$ and has size $2n+5$.
    Let $S'\coloneqq S\cap V(G)$. Since $y \notin S$, it is immediate
    that $S'$ is an {\lds} of $G$ (as it inherits this property from $S$
    with respect to $H$).  Moreover, $S'$ is a minimum {\lds} of $G$,
    otherwise, we could take a smaller {\lds} of $G$, which together
    with $\hat{S}$ would yield an {\lds} of~$H$ smaller than~$S$, a
    contradiction.  Thus, $|S'| = \gamma^{LD}(G)$, and
    $|S| = 2n + 5 + \gamma^{LD}(G)$.

   \medskip

    \noindent\textbf{Case (b)} When  $y \in S$.
    Let $S'\coloneqq S \cap V(G)$.  We claim that $z \notin
    S$. Indeed, suppose $z\in S$.  Since $S$ is a minimum {\lds}
    of~$H$, by Lemma~\ref{lemma:lds-butterfly} we know that $S$ contains
    $x$. Thus, $S\setminus\{y\}$ is also an {\lds} of $H$.  But this
    contradicts the minimality of~$S$.  Thus, for now we have that $y$
    is in $S$ and $z$ is not in $S$. Let
    $Z \coloneqq (S \setminus \{y\}) \cup \{z\}$. 
   
    We claim that $Z$ is an \lds{} of $H$. For that, it suffices to
    show that $S'\cup \{z\}$ is an {\lds} of $G$.  Clearly,
    $S' \cup \{z\}$ is a dominating set of $G$.  To verify that
    $S' \cup \{z\}$ has the locating property in~$G$, note that every
    pair of vertices $u, v \notin S'$, which are in $G - z$, have
    distinct neighbourhoods in $S'$ (a property inherited from
    $S$). Thus, such pairs $(u, v)$ continue to have distinct
    neighbourhoods in $S' \cup \{z\}$. As $S' \cup \{z\}$ is an {\lds}
    of $G$, it follows that $Z$ is an \lds{} of $H$.  Since
    $|Z| = |S|$, we have that $Z$ is a minimum \lds{} of $H$. Since
    $|Z \cap V(\hat{H})| = |S \cap V(\hat{H})| = 2n+5$ and
    $y\notin Z$, by the previous case, we conclude that $Z\cap V(G)$
    is a minimum {\lds} of $G$ and $|S|= |Z| = 2n+5 + \gamma^{LD}(G)$.
\end{proof}


\section{The infinite periodic graph $\mathbf{G^\infty}$}
\label{sec:G-infinite}

\begin{definition} [\textbf{The graph}  $\mathbf{G^{\infty}}$] \label{def:inf_graph}
    Let $G$ be a finite connected graph with $n\geq 5$ vertices, and let $z$ be 
    an arbitrary vertex of $G$. 
    The infinite graph $G^{\infty}$ consists of infinitely many copies $H_i$,
    $i \in \mathbb{Z}$, of the $G_i$-gadget $H(G_i)$, where $G_i$ is isomorphic to $G$.      
    The vertices of each $H_i$ are labelled precisely as in $H(G)$, but with an 
    additional superscript~$i$.  For each $i\in \mathbb{Z}$, the subgraph  $H_i$ is connected 
    to $H_{i+1}$ by adding the edge $\{v^i_{5n}, b^{i+1}\}$. 
    See Figure~\ref{fig:infinite_graph}.
\end{definition}

\begin{figure}[ht]  
    \centering
    \includegraphics{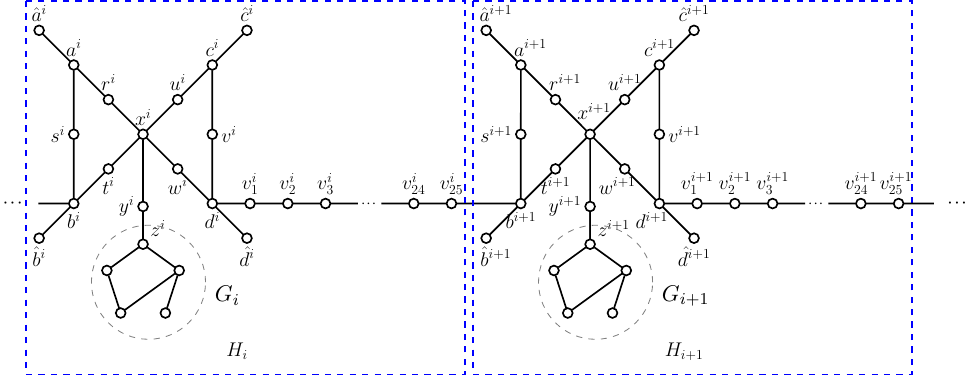}
    \caption{The infinite graph $G^{\infty}$ obtained by
    chaining infinite copies of $H(G)$.}
    \label{fig:infinite_graph}
\end{figure}

\begin{lemma}  
\label{lemma:upper-bound}
    For any finite connected graph $G$ of order $n\geq 5$, the infinite graph 
    $G^{\infty}$ constructed as in Definition~\ref{def:inf_graph} has an \lds{} 
    $S$ such that $|S\cap V(H_i)| = 2n + 5 + \gamma^{LD}(G)$ for every $i \in \mathbb{Z}$.
\end{lemma} 

\begin{proof}
    Let $\tilde{S_i}$ be a minimum \lds of $G_i$. In this case, 
   by Lemmas~\ref{lemma:lds-butterfly} and~\ref{lemma:join-graphs}, the set
    \[
        S_i = \{a^i, b^i, c^i, d^i, x^i\} \cup \big\{ v^i_{5j + 2}, v^i_{5j + 4} : 
        j = 0, \ldots, n - 1\big\} \cup \tilde{S_i},
    \]
    is a minimum \lds of $H_i$, and hence $|S_i| = 2n + 5 + \gamma^{LD}(G)$.

    Now let $S = \bigcup_{i\in \mathbb{Z}}S_i$.  Then
    $|S\cap V(H_i)| = 2n+ 5 + \gamma^{LD}(G)$ for every
    $i\in \mathbb{Z}$.  It remains to prove that $S$ is indeed an \lds
    of $G^{\infty}$.  Note that each~$S_i$ is an \lds within its
    corresponding gadget~$H_i$.  We observe that, by
    Lemma~\ref{lemma:lds-butterfly}, $S_i$ intersects the Butterfly,
    say $B_i$ (of $H_i$), precisely in the set
    $\{a^i, b^i, c^i, d^i, x^i\}$, so the vertex $y^i$ is dominated by
    $x^i$ and is distinguished by $S_i$ from the other neighbours of
    $x^i$ in $B_i$.

    Let us analyse  what happens at the boundary of $H_i$ and $H_{i+1}$
    (the boundary of $H_{i-1}$ and $H_{i}$ is equivalent).  Since the
    only edge joining $H_i$ and $H_{i+1}$ is $v^i_{5n}b^{i+1}$, and
    $b^{i+1}\in S$, we only need to check pairs of vertices adjacent
    to $b^{i+1}$ which are in different gadgets. This means that it
    suffices to check pairs in which one of the vertices is
    $v^i_{5n}$. But since $v^i_{5n}$ is neighbour to $v^i_{5n-1}$, and
    $v^i_{5n-1}\in S_i$, all pairs containing $v^i_{5n}$ are
    distinguished by $S$. Thus, $S$ has the locating property.
    As $S$ is a  dominating set, we conclude that $S$ is an \lds
    of $G^{\infty}$.

    We observe that in the definition of $S_i$ we could have chosen
    differently the $2n$ vertices in the path $P^i$, but for sure
    either $v_{5n-1}^i$ or $v_{5n}^i$ has to be chosen. When 
    $v_{5n}^i$ is chosen, the conclusion that $S$ has the
    locating property is immediate. 
\end{proof}

We are now interested in the minimum density of an \lds of
$G^{\infty}$. To do this, we prove first that $G^{\infty}$ is a
well-behaved graph, that is,  $G^{\infty}$  has the Slow Growth property. 

\begin{definition}
  An infinite graph $G$ has the Slow Growth property {\rm
    (SG}-property{\rm )}
if $G$ is connected and has a vertex $s$ such that
 $ \lim_{r \to \infty} {|N_{r+1}[s]|}/{|N_{r}[s]|} = 1$.
\end{definition}

\begin{theorem}
    For any finite connected graph $G$ of order $n\geq 5$, the infinite graph 
    $G^{\infty}$ constructed as in Definition~\ref{def:inf_graph}
    has the {\rm SG}-property.
\end{theorem}

\begin{proof}
    Take $s \coloneqq x^i$ for some $i \in \mathbb{Z}$.
    Since $N_r[s] \subset N_{r+1}[s]$ for every $r \geq 0$, we have 
    that $|N_{r+1}[s]|/|N_r[s]| \geq 1$. 
    Moreover, analysing a breadth-first search tree of $G^{\infty}$ rooted at $s$, 
    it is not hard to see that the largest growth of this tree from a layer at
    distance~$r$ from~$s$ to the next layer $r+1$ may be of order at
    most $2n + 6$, that is, $|N_{r+1}[s]| \leq |N_{r}[s]| + C$, where
    $C = 2n+6$. Therefore, $1 \leq |N_{r+1}[s]|/ |N_r[s]| \leq 1 +
    C/|N_r[s]|$, and hence, \hbox{$\lim_{r \to \infty} {|N_{r+1}[s]|}/{|N_r[s]|} = 1$.}
  \end{proof}

This property was introduced by Sampaio et al.~\cite{SampaioSW24}. One
of the consequences of this property (there are others) is that we may
use the following lemma.

\begin{lemma}[Sampaio et al.~\cite{SampaioSW24}] \label{lem:SG-property} 
    Let $\hat{G}$ be an infinite connected graph with bounded maximum
    degree that satisfies the {\rm SG}-property. Let $\ell, c, c', \delta$ be
    positive integers, and let $C$ be a subset of
    $V(\hat{G})$. If $\hat{G}$ can be
    partitioned into finite sets $V_1,V_2, \ldots$  of size $\ell$ such that
    $c\leq |V_i\cap C|\leq c'$ for $i\in \mathbb{Z}$, and the distance between
    any two vertices of $V_i$ is at most $\delta$, then the density of $C$
    in $\hat{G}$  satisfies $c/\ell \leq  d(C,\hat{G}) \leq c'/\ell$.
\end{lemma}

Note that, as $G$ is a finite graph of order $n\geq 5$, then 
 $G^{\infty}$ has maximum degree at most~$n$, and the
diameter of each $H_i$ is at most~$6n+4$. Thus, applying
Lemma~\ref{lem:SG-property} to $G^{\infty}$ we have:

\begin{lemma} \label{cor:Ginfinito} Let $G$ be a finite connected
  graph with $n\geq 5$ vertices, and let $G^{\infty}$ be the infinite
  connected graph mentioned in Definition~\ref{def:inf_graph}.  Let
  $c, c', \ell$ be positive integers.  Consider the natural partition
  of $V(G^{\infty})$ into finite parts $V(H_i)$ for $i\in \mathbb{Z}$,
  where $\ell = |V(H_i)|= 6n+16$, and each $H_i$ is a periodic
  subgraph of $G^{\infty}$.  Let $S$ be an \lds{} of
  $G^{\infty}$ such that $c\leq |V(H_i)\cap S|\leq c'$ for
  $i\in \mathbb{Z}$.  Then $c/\ell\leq d(S,G^{\infty})\leq c'/\ell$.
\end{lemma}

\begin{lemma}
    \label{lemma:min-density}
    If $G$ is a finite connected graph of order $n\geq 5$, then the
    minimum density of an \lds in  $G^{\infty}$ is $d^*(G^{\infty})   = (2n + 5 + \gamma^{LD}(G))/(6n + 16)$. 
\end{lemma}

\begin{proof}
    Let $\dot{S}$ be an \lds{} of $G^{\infty}$ as defined in the proof
    of Lemma~\ref{lemma:upper-bound}. Consider the partition of
    $V(G^{\infty})$ into finite parts $V(H_i)$ for $i \in \mathbb{Z}$ of
    order $\ell = |V(H_i)| = 6n + 16$, where each $H_i$ is a periodic
    subgraph of $V(G^{\infty})$. 
    Then, since $|V(H_i) \cap \dot{S}| =  2n + 5 + \gamma^{LD}(G)$,
    by Lemma~\ref{cor:Ginfinito}, we have that $d(\dot{S},G^{\infty})= 
     (2n + 5 + \gamma^{LD}(G))/(6n + 16)$.
    This gives us an upper bound for $d^*(G^{\infty})$, that is,
    \begin{equation}
        d^*(G^{\infty}) \leq (2n + 5 + \gamma^{LD}(G))/(6n + 16). \label{eq:1}      
    \end{equation}

    Now, let $S$ be a minimum-density \lds of $G^{\infty}$ and let
    $H_i$ be a subgraph of $G^{\infty}$ such that $|S \cap V(H_i)|$ is
    the smallest possible among all gadgets. Let
    $S^i \coloneqq S \cap V(H_i)$. We shall prove that
    $|S^i|\geq 2n + 5 + \gamma^{LD}(G)$.

    Since $H_i$ has the same local structure as the gadget of
    Definition~\ref{def:gadget}, and because of the choice of~$S^i$,
    by the same argument used in the proof of
    Lemma~\ref{lemma:lds-butterfly}, accounting for the additional
    edges $b^{i+1}v_{5n}^i$ and $b^iv_{5n}^{i-1}$, which do not affect
    the conclusion, the intersection of $S^i$ with the Butterfly of
    $H_i$ is exactly the set $\{a^i, b^i, c^i, d^i, x^i\}$.


    Now, according to~\cite{BertrandCHL04}, in each
    set of $5$ consecutive vertices of a path, we must have at least
    two in any \lds{}.  However, in $H_i$ the path
    $P^i = (v^i_1, v^i_2, \ldots, v^i_{5n})$ has its initial vertex
    connected to $d^i$ (which is in $S^i$) and its end vertex
    connected to $b^{i+1}$. A careful analysis shows that
    $v_2^i\in S^ i$, and although there is not a unique choice for
    $S^i \cap V(P^i)$, we conclude that this intersection contains at
    least $2n$ vertices (including either $v_{5n-1}^i$ or $v_{5n}^i$).
    Thus, $|S^i\setminus V(G_i + z^iy^i)| \geq 2n + 5$.  Now we
    analyse two cases. 
    
    \textbf{Case~(a)} If $S^i$ does not contain $y^i$, then
    $S^i \cap V(G_i)$ is a minimum {\lds} of~$G_i$ and therefore,
    $|S^i| \geq 2n+ 5 + \gamma^{LD}(G)$.  \textbf{Case~(b)} If $S^i$
    contains $y^i$, since $S^i$ contains $x^i$, then $z^i \notin S^i$,
    otherwise we could delete $y^i$ from $S^i$, a contradiction to the
    choice of $H_i$.  Let $Z = (S^i \setminus \{y^i\}) \cup \{z^i\}$.
    Then $Z \cap V(G_i)$ is a minimum {\lds} of $G_i$ (if there
    existed a smaller {\lds}, it would contradict the choice of
    $H_i$). Then
    $|Z \cap V(G_i)| = |S^i \cap V(G_i)| + 1 = \gamma^{LD}(G_i)$, and
    hence, $|S^i\cap V(G_i))| = \gamma^{LD}(G_i) - 1$. Thus, 
     $|S^i| \geq 2n + 5 + |\{y\}| + \gamma^{LD}(G_i) - 1 = 2n + 5 +
    \gamma^{LD}(G)$.  Therefore,  in both cases we have that
    $|S^i| \geq 2n + 5 + \gamma^{LD}(G)$, and, from
    Lemma~\ref{cor:Ginfinito}, we get that
    $d^*(G^{\infty}) = d(S,G^{\infty}) \geq (2n + 5 +
    \gamma^{LD}(G))/(6n + 16)$.
   
    This lower bound combined  with the inequality~\eqref{eq:1} gives 
    the  desired equality. 
  \end{proof}

\section{Reducing the {\MinCard} \lds to {\MinDen} \lds in an infinite graph}
\begin{theorem}\label{theo:NP-hardness}
    The {\MinDen} \lds problem is {\NP}-hard in 
    infinite $\mathbb{Z}$-periodic graphs with a finite period.
\end{theorem}
\begin{proof}
    We present a polynomial-time reduction from the {\MinCard} 
    \lds problem in a finite graph $G$ to the {\MinDen} \lds
    problem in the infinite $\mathbb{Z}$-periodic graph $G^{\infty}$.

    To this end, let $G = (V, E)$ be a connected finite graph
    with~$n\geq 5$ vertices (given as an instance of the first problem).
    Given $G$, we choose an arbitrary vertex $z$ in $G$, and define the
    graph $G^{\infty}$ as in Definition~\ref{def:inf_graph}.  Since
    $G^{\infty}$ is a $\mathbb{Z}$-periodic graph with a finite period
    that is a subgraph isomorphic to the gadget $H(G)$, it is immediate
    that the description of $G^{\infty}$ takes time polynomial in the
    size of $G$ (as it suffices to describe one subgraph $H(G_i)$ and
    its connection to $H(G_{i+1})$).  By Lemma~\ref{lemma:min-density}, we have that 
    %
    $$ d^*(G^{\infty}) =  (2n + 5 + \gamma^{LD}(G))/ (6n + 16).$$
    
    By Lemma~\ref{lemma:upper-bound}, we know that $G^{\infty}$ has an
    \lds{} $S$ such that $|S\cap V(H_i)| = 2n+ 5 + \gamma^{LD}(G)$ for
    each $i\in \mathbb{Z}$.  Thus, an optimal solution for
    $G^{\infty}$ can be obtained from an optimal solution for $H_i$,
    for any $i\in \mathbb{Z}$. Hence, an optimal solution for
    $G^{\infty}$ also has a description that is polynomial in the
    size of $G$. So, for some $i\in \mathbb{Z}$, let $S_i$ be an
    optimal solution for $H_i$.  By Lemma~\ref{lemma:join-graphs},
    if $y^i\notin S_i$, then $|S_i\cap V(G_i)| = \gamma^{LD}(G_i)$,
    thus, $S_i\cap V(G_i)$ defines a minimum \lds for $G$. If
    $y^i\in S_i$, then by this same lemma, if we take
    $Z_i = (S_i \setminus \{y_i\}) \cup \{z_i\}$, then $Z_i$ is a minimum
    \lds of $H_i$ and $Z_i\cap V(G_i)$ defines a minimum \lds{} of
    $G_i$. Hence, $Z_i$ defines a minimum \lds of $G$.
    
    Note that the value of $\gamma^{LD}(G)$ can also be obtained
    from the equality above, as
    $\gamma^{LD}(G) = d^*(G^{\infty}) \cdot (6n + 16) - 2n - 5$.  We
    may say that $G$ has an {\lds} of cardinality at most $K$, if
    and only if $G^{\infty}$ has an {\lds} with density at most
    $\widehat{K} = (2n+5 + K)/(6n+16)$.

    As $G^{\infty}$ is an infinite $\mathbb{Z}$-periodic graph with
    a finite period, the proof of the theorem is complete.
\end{proof}

In 2015, Foucaud~\cite{Foucaud15} proved that the {\MinCard} \lds
problem is NP-hard even on subcubic planar bipartite graphs. As a
consequence, we obtain the following stronger result.
  
\begin{corollary}
    The {\MinDen} {\lds} problem is {\NP}-hard in 
    infinite $\mathbb{Z}$-periodic planar bipartite graphs with maximum
    degree~$5$ and with a finite period. 
\end{corollary}

\section{Concluding remarks}

We believe that the approach presented here for the \lds problem could
be useful for proving NP-hardness results for similar problems on
infinite periodic graphs.

For an updated and comprehensive bibliography on locating-dominating
sets and related topics, we refer the reader to the collection
maintained by Jean~\cite{Jean24}, which builds upon the long-standing
work of Lobstein.

\bibliographystyle{elsarticle-num}
\bibliography{lds-hardness}

\end{document}